\documentclass[a4paper,12pt]{article}
\usepackage{latexsym}
\usepackage{amsmath}
\usepackage{amssymb}
\usepackage{tikz}
\usepackage{color}

\usepackage{theorem}
\newtheorem{theorem}{Theorem}[section]
\newtheorem{proposition}[theorem]{Proposition}
\newtheorem{lemma}[theorem]{Lemma}

\newtheorem{MTA}{Theorem A}
\newtheorem{MTB}{Theorem B}

\theorembodyfont{\rmfamily}
\newtheorem{proof}{\textmd{\textit{Proof.}}}

\newtheorem{remark}[theorem]{Remark}

\makeatletter

\@addtoreset{equation}{section}
\makeatother

\newcommand{\qedd}{\hfill \Box}

\newcommand{\Sph}{\ensuremath{\mathbb{S}}}

\title{ On the Structure of the Cut Locus in Certain Classes of Warped Product Manifolds
}
\author{Nathaphon Boonnam}
\date{}
\begin{document}

%%%%%%%%%%%%%%%%%%%%%%%%%%%%%%%%%%%%%%%%%%%%%%%%%%%

\maketitle

\begin{abstract} In this paper, we determine the structure of the cut locus for a class of complete warped product manifolds $M=(R^1\times S^1,dt^2+m(t)^2d\theta^2)$, where the warping function $m(t)$ is positive and even, under the assumption that   
there exists a positive number $t_1$ such that  
the Gaussian curvature  is
nonpositive on the domain  
 $t^{-1}\left((t_1,\infty) \right)$ and  is
 decreasing 
along the subarc 
$t^{-1}([0,t_1])\cap\theta^{-1}(0)$ of the upper half meridian.

\end{abstract}

\section{Introduction}

The study of cut loci on Riemannian manifolds represents one of the most important areas in differential geometry
 since the pioneering work of H. Poincaré \cite{P}.
J. Hebda \cite{H} established that the cut locus of a point on a complete 2-dimensional Riemannian manifold  forms a local tree structure (see also   \cite{ShT} and \cite{SST}). Notice that the cut locus structure for a surface of revolution can be complicated. In fact, Gluck and Singer \cite{GS} constructed a 2-sphere of revolution which has a non-triangluable cut locus.

D. Elerath \cite{E} would be the  first researcher who studied the cut locus for a classical surface of revolution.
He determined the cut locus of a point  for a
flattening surface of revolution that he calls. 
He noticed the monotonicity of the Gaussian curvature and defined a flattening surface of revolution as a classical surface of revolution that has the property of the decline of the Gaussian curvature along a meridian.  Typical examples of flattening surfaces are 
a paraboloid and a 2-sheeted hyperboloid.
This idea is originated in the study of a 2-sheeted hyperboloid by H. von Mangoldt \cite{Ma} in 1881.
Since it was achieved before the introduction of the concept of the cut locus  by  H. Poincar\'e
 \cite{P},  he did not try to determine the cut locus structure of  2-sheeted hyperboloids of revolution. However, he found a sufficient condition for a point on a 2-sheeted hyperboloid to be a pole.
Well-known and very familiar classical surfaces of revolution
have the property  that the Gaussian curvature is monotone along a subarc of a meridian. After Elerath's work,
the cut locus structure was determined for the ellipsoids of revolution \cite{ST} and  the tori of revolution \cite{GMST} in  3-dimensional Euclidean space 
by making use of this property and the Rauch comparison theorem.
In fact, the structure was determined for  certain classes of abstract surfaces of revolution containing such classical surfaces (see \cite{BCST}, \cite{ST}, \cite{TASY}, \cite{TASY1}, and \cite{T2}).

  It was proved in \cite{ST} that if the Gaussian curvature is monotone along a subarc of a subarc of a meridian, then  
the half period function defined 
in \eqref{eq2.2} is monotone and in \cite{BCST},
the cut locus structure was  determined for a 2-sphere of revolution if the half period function is decreasing. 
Notice that for the family of 2-spheres of revolution $(\Sph^2, dr^2+m_\lambda(r)^2d\theta^2)$ given in \cite{BCST}, the Gaussian curvature of the 2-sphere is not monotone along a meridian from each pole to the point on the equator if $\lambda>2. $ 
Thus, the monotonicity  is not always a suitable assumption to determine the cut locus structure (see also  \cite{TASY} and \cite{TASY1}).

P. Chitsakul \cite{C1} determined the cut locus structure for  a  cylinder of revolution under the assumption that the Gaussian curvature is decreasing along the upper half meridian and   generalized in \cite{C2}.

In this paper, we will  prove the following theorem.

\begin{MTA} Let $(M,ds^2)$ be a complete Riemannian manifold $R^1\times S^1$ with a warped product metric $ds^2=dt^2+m(t)^2d\theta^2$ of the straight line $(R^1,dt^2)$ and  the unit circle $(S^1,d\theta^2)$ with a warping function $m.$ 
Suppose that $M$ has a reflective symmetry with respect to the equator $t=0$ and
there exists a positive number $t_1$ such that  
the Gaussian curvature  is
nonpositive on  the domain  
 $t^{-1}\left((t_1,\infty) \right)$ and  is
 decreasing 
along the subarc 
$t^{-1}([0,t_1])\cap\theta^{-1}(0)$ of the upper half meridian.  
Then the cut locus of a point $q\in\theta^{-1}(0) $ is given as follows:
\begin{enumerate}
    \item If $|t(q)|<t_0:=\sup\{t>0|m'(t)<0\}$, then
    $$C_q=\theta^{-1}(\pi)\cup\left( t^{-1}(-t(q))\cap\theta^{-1}\left([\varphi(m(t(q))),   2\pi-\varphi(m(t(q)))]\right)\right),$$
    where $\varphi(\nu):=2\int_0^{\xi(\nu)}\frac{\nu}{m(x)\sqrt{m(x)^2-\nu^2}}dx$ and $\xi(\nu):=(m|_{[0,t_0)})^{-1}(\nu)$ for $\nu\in(m(t_0),m(0)),$
    \item If $|t(q)|\geq t_0$ or if $\varphi(m(t(q)))\geq\pi,$ then
    $C_q=\theta^{-1}(\pi),$ which is the meridian  opposite to $q.$  
\end{enumerate}
\end{MTA}

\begin{remark}\label{rem1.1}
If the Gaussian curvature of the manifold $M$ is decreasing along the upper half meridian, then by
\cite[Theorem 2.2.1]{SST} 
 the Gaussian curvature is nonpositive outside of a bounded domain $t^{-1}((-t_1,t_1))$ for some $t_1>0.$
 Under the assumption
 that the Gaussian curvature of the manifold $M$ is decreasing along the upper half meridian,  the conclusion of {\bf Theorem A} was proved   
in \cite{C1}. 
\end{remark}

\section{Preliminaries} First, let us review here the definition  of the cut locus.
Let $M$ denote a complete Riemannian manifold  and $\gamma:[0,a]\to M$  a minimal geodesic segment emanating from a point $p=\gamma(0)$ to $\gamma(a)$. 
 The end point $\gamma(a)$ of the geodesic segment $\gamma$ is called a  \emph{cut point} of $p$ along $\gamma$ if any geodesic extension $\tilde{\gamma}:[0,a+\varepsilon]\to M$ of $\gamma$ is not minimal anymore.
The \emph{cut locus} $C_p$ of the point $p$  is the set of all cut points of $p$ along minimal geodesic segments emanating from the point $p.$

From now on,
 $(M,ds^2)$ denotes  a complete Riemannian manifold $R^1\times S^1$ with a warped product metric $ds^2=dt^2+m(t)^2d\theta^2$ 
 of the real line $(R^1,dt^2)$  and the unit circle $(S^1,d\theta^2).$ We assume that
 $m:R\to (0,\infty)$ is a smooth even function and that there exists a positive number $t_1$  such that $m''(x)\geq 0 $ on 
 $[t_1,\infty)$ and $m''(x)/m(x)$ is increasing on $[0,t_1].$  We refer to \cite{SST} on the behavior of geodesics on a surface of revolution. 
 Then the Gaussian curvature $G(p)$ at a point $p\in M$ equals 
\begin{equation}\label{eq2.1}
G(p)=-\frac{m''(t(p))}{m(t(p))}.
\end{equation}
Therefore, 
the property that the
Gaussian curvature  is
nonpositive on   the domain  $t^{-1}\left( 
(t_1,\infty) \right)$
means that $m''(x)\geq 0$ on $[t_1,\infty)$
 and the one that the curvature function is
 decreasing 
along the subarc 
$t^{-1}([0,t_1])\cap\theta^{-1}(0)$ means that
the function  $m''(x)/m(x) $ is increasing
on $[0,t_1].$
Incidentally,
the property that $M$ is reflectively symmetric with respect to the equator $t=0$ is equivalent to the one that  $m(x)$ is an even function.
  
  In order to prove {\bf Theorem A}, it is sufficient to prove the following theorem.
  
  \begin{MTB}
  Let  
  $(\widetilde M,d\tilde s^2)=(R^1\times R^1,d\tilde{t}^2+m(\tilde{t})^2d\tilde{\theta}^2)$ denote the universal covering space of 
    the complete Riemannian manifold $(M, ds^2)$  with a warped product metric $ds^2=dt^2+m(t)^2d\theta^2$ 
 of the real line $(R^1,dt^2)$ and
 the unit circle $(S^1,d\theta^2).$
  Suppose that  there exists a positive number $t_1$ such that $m''(x)\geq 0$ on 
 $[t_1,\infty)$ 
  and the  function $m''(x)/m(x)$ is increasing on $[0,t_1].$
  If $m$ is even,
  then the structure of cut locus $C_q$ of a point $q\in\tilde\theta^{-1}(0)$ of $\widetilde M$ is given as follows:
  
\begin{enumerate} 
\item If $|\tilde t(q)|<t_0:=\sup\{t>0|m'(t)<0\}$, then
  $$C_{ q}=\{p\in\widetilde M|\tilde t( p)=-u,|\tilde\theta( p)|\geq\varphi(m(u)))\},$$
  where $u:=\tilde t( q)
$.  
\item If $|\tilde t(q)|\geq t_0,$
  $C_{q}$ is empty. 
   
\end{enumerate}
  \end{MTB} 
\begin{remark}
By {\bf Theorem B},
the cut locus of any point of the upper half plane $( (0,\infty)\times R^1, d\tilde t^2+m(\tilde t)^2d\tilde\theta^2)$ with totally geodesic boundary is empty. This property plays  a role of a suitable model surface  
when we apply  the Toponogov comparison theorem to open geodesic triangles standing on the totally geodesic boundary of a Riemannian manifold
(see \cite{KT1} and \cite{KT2}).
\end{remark}

  From now on,  
  $(\widetilde M,d\tilde s^2)=(R^1\times R^1,d\tilde{t}^2+m(\tilde{t})^2d\tilde{\theta}^2)$ denotes the universal covering space of 
    the complete Riemannian manifold $(M,ds^2).$  
  For each $\nu\in[0,m(0)],$ let $\gamma_{\nu}:[0,\infty)\to \widetilde M$ and 
$\beta_{\nu}:[0,\infty)\to \widetilde M$ denote the unit speed geodesic segments emanating from a common point $p\in \tilde t^{-1}(0)$ with Clairaut constant $\nu$ such that
$(\tilde t\circ\gamma_{\nu})'(0)\leq 0\leq (\tilde t\circ\beta_{\nu})'(0).$
Note that $\gamma_{\nu}$ and $\beta_{\nu}$ are reflectively symmetric
with respect to the parallel arc $\tilde t=0,$ since $m$ is even.
Then for each $\nu\in(m(t_0),m(0)),$ 
where $t_0:=\inf\{t>0| m'(t)<0\}$ and $m(t_0):=\inf m$ when $t_0=\infty,$ 
$\gamma_{\nu}$ (respectively $\beta_{\nu}$) intersects the parallel arc $\tilde t=0$ again at a point $\gamma_{\nu}(l(\nu))=\beta_{\nu}(l(\nu)).$
If we notice that $\gamma_{\nu}$ (respectively $\beta_{\nu}$) is tangent to the parallel arc $\tilde t=-\xi(\nu)$   (respectively $\tilde t=\xi(\nu)$), where
$\xi(\nu)=(m|_{[0,t_0)})^{-1}(\nu),$ before their second intersections with $\tilde t=0,$
then from \cite[Proposition 7.1.3]{SST}
it follows that
\begin{equation*}
\tilde\theta(\gamma_{\nu}(l(\nu)))-\tilde\theta(\gamma_{\nu}(0))=\varphi(\nu),
\end{equation*}
where \begin{equation}\label{eq2.2}
\varphi(\nu):=2\int_0^{\xi(\nu)}\frac{\nu}{m(x)\sqrt{m(x)^2-\nu^2}}dx, 
\end{equation} 
and that the length $L(\gamma_{\nu}|_{[0,l(\nu)]})=l(\nu)$ of the segment $\gamma_{\nu}|_{[0,l(\nu)]}$ is equal to
\begin{equation}\label{eq2.3}
L(\gamma_{\nu}|_{[0,l(\nu)]})=2\int_0^{\xi(\nu)}\frac{m(x)}{\sqrt{m(x)^2-\nu^2}}dx.
\end{equation}

 \begin{lemma}\label{lem2.1}
 The derivative $l'(\nu)$ of $l(\nu)$
 equals
 \begin{equation}\label{eq:lem2(1)}
 l'(\nu)=\nu\varphi'(\nu)
 \end{equation}
 on $(m(t_0),m(0)).$
 \end{lemma}
\begin{proof}
If we note that $\frac{m}{\sqrt{m^2-\nu^2}}=\frac{\sqrt{m^2-\nu^2}}{m}+\frac{\nu^2}{m\sqrt{m^2-\nu^2}}$ holds, we get, by \eqref{eq2.2} and \eqref{eq2.3},
$$l(\nu)=2\int_0^{\xi(\nu)}\frac{\sqrt{m(x)^2-\nu^2}}{m(x)}dx+\nu\varphi(\nu).$$
By differentiating the above equation, we have $l'(\nu)=\nu\varphi'(\nu).$
$\qedd$
\end{proof}
Here we give an elementary proof of the claim stated in Remark \ref{rem1.1}.
\begin{lemma}\label{lem2.2}
There exists a positive number $t_1$ such that $m''(t_1)\geq0.$
\end{lemma}
\begin{proof}
By supposing that $m''(x)<0$ for all $x\in[0,\infty),$ we will get a contradiction. 
Since  $m'(0)=0$ and 
 $m'(x)$ is strictly decreasing on $[0,\infty),$ 
$0=m'(0)>m'(1)>m'(x)$ for all $x>1.$
Hence,  we get 
$m(x)-m(1)=\int_1^x m'(x)dx<\int_1^x m'(1)dx$ and  $\lim_{x\to\infty} m(x)=-\infty.$ This is a contradiction. 
$\qedd$
\end{proof}
\section{Decline of the half period function}

%%%%%%%%%%%
By imitating the proof of \cite[Lemma 1]{TASY}, we get
\begin{lemma}\label{lem3.1}
     The half period function $\varphi(\nu)$
     equals  
    \begin{equation}
      \varphi(\nu)=-2\int_0^\infty A\circ \xi\left(\sqrt{u(\tau,\nu)}\right)/({a^2\tau^2+1}) d\tau,
       \label{eq3.1}
    \end{equation}
  where $A(x):={\sqrt{a^2-m(x)^2}}/{m'(x)},$   $a:=m(0)$ and $u(\tau,\nu):={\nu^2(a^2\tau^2+1)}/{(\tau^2\nu^2+1)}$.
\end{lemma}
%%%%%%%%%%%%%%%%%%%%%%%%%%%%%%%%%%%

\begin{proposition}\label{prop3.2}
    If $m''(x)/m(x)$ is increasing on $(0,t_1)$ and $m''(t)\geq 0$ on $[t_1,\infty)$, then $A(x)$ is decreasing on $(0,t_0),$ and 
$\varphi(\nu)$ is decreasing on $(m(t_0),m(0))$. 
\end{proposition}
\begin{proof}
 It is easy to check that
    \begin{equation}\label{eq:prop3.2(1)}
        A'(x) 
        =\frac{-m(x)}{m'(x)^2\sqrt{a^2-m(x)^2}}F(x),
    \end{equation}
   where
    \begin{equation}\label{eq:prop3.2(2)}
   F(x):=m'(x)^2+\frac{m''(x)}{m(x)}(a^2-m(x)^2)
   \end{equation} 
 Hence, by Lemma \ref{lem3.1} it  suffices to prove that 
   $F(x)\geq 0$ on $[0,t_0].$
   Since
   $
        F'(x)=\left({m''(x)}/{m(x)}\right)'(a^2-m(x)^2)
    $
    and since $m''(x)/m(x)$ is increasing on $(0,t_1)$, we get
        $$F'(x)\geq 0$$ 
    on $[0,t_2],$  where $t_2:=\min\{t_0,t_1\}.$  This means that $F(x)$ is increasing on $[0,t_2]$ and then
        $$F(x)\geq F(0)=0$$
    on $[0,t_2].$
    Thus, if $t_2=t_0,$  $F(x)\geq0$ on $[0,t_0]$ and if $t_2=t_1<t_0$, by  
     \eqref{eq:prop3.2(2)}, it
     is  
     clear that $F(x)\geq0$ on $[t_1,t_0].$
     Therefore, by  \eqref{eq:prop3.2(1)},
     $A'(x)\leq 0$ on $(0,t_0)$ and 
     by Lemma \ref{lem3.1},   we can conclude that $\varphi(\nu)$ is decreasing on $(m(t_0),m(0))$.
     $\qedd$
\end{proof}
%(see \cite[Theorem 7.1.1]{SST}).

\section{The cut locus of a point in $\widetilde M$}
Choose any point $q\in\tilde t^{-1}\left((-t_0,t_0)\right)\cap\tilde\theta^{-1}(0)$
and put $\tilde t(q)=:u$.
For each $\nu\in(m(t_0),m(u)]$, let $\gamma^{(u)}_\nu:[0,\infty)\to\widetilde M$ (respectively $\beta^{(u)}_\nu:[0,\infty)\to\widetilde M$) denote the unit speed geodesic emanating from the point $q$ on the parallel arc $\tilde t=u$ with Clairaut constant $\nu$ such that $(\tilde t\circ\gamma^{(u)}_\nu)'(0)\leq 0$ (respectively $(\tilde t\circ\beta^{(u)}_\nu)'(0)\geq 0$).

\begin{lemma}\label{lem4.1} For each $\nu\in(m(t_0),m(u)),$ 
    both geodesics $\gamma_{\nu}^{(u)}\:and\:\beta_{\nu}^{(u)}$ intersect
    at 
    the same point $(\tilde t,\tilde\theta)^{-1}(-u,\varphi(\nu)).$ 
\end{lemma}

\begin{proof}
Let us review here the behavior of both geodesic $\gamma^{(u)}_{\nu}$ and $\beta^{(u)}_{\nu}.$
The geodesic
$\gamma^{(u)}_{\nu}$ (respectively
$\beta^{(u)}_{\nu}$) is tangent to the parallel arc $\tilde t=-\xi(\nu)$ (respectively $\tilde t=\xi(\nu)$) and intersects the arc 
$\tilde t=-u$ at a point $\gamma^{(u)}_\nu(s_1)$ (respectively $\beta_{\nu}^{(u)} (s_2)$).
    From
\cite[Proposition 7.1.3]{SST} it follows that 
    
\begin{equation}\label{eq:lemma4.1(1)}
    \tilde\theta(\gamma^{(u)}_{\nu})(s_1)=-
\int_u^{-\xi(\nu)}f(x,\nu)dx
+\int_{-\xi(\nu)}^{-u}
f(x,\nu)dx,
\end{equation}
\begin{equation}\label{eq4.2}
\tilde\theta(\beta^{(u)}_{\nu})(s_2)=
\int_u^{\xi(\nu)}f(x,\nu)dx
-\int^{-u}_{\xi(\nu)}
f(x,\nu)dx,
\end{equation}
where $f(x,\nu):=\frac{\nu}{m(x)\sqrt{m(x)^2-\nu^2}}.$
If we note that $f(-x,\nu)=f(x,\nu)$, we get
\begin{equation}\label{eq:lemma4.1(2)}
    \tilde\theta(\gamma_{\nu}^{(u)}(s_1))=\varphi(\nu)
    =\tilde\theta(\beta_{\nu}^{(u)}(s_2)).
   \end{equation} 
   Therefore, both geodesics 
intersect the common point $(\tilde t,\tilde\theta)^{-1}(-u,\varphi(\nu)).$
$\qedd$
\end{proof}

\begin{lemma}\label{lem4.2} 
    The two geodesic segments $\gamma^{(u)}_\nu|_{[0,s_1]}$ and $\beta^{(u)}_\nu|_{[0,s_2]}$ have the same length $l(\nu)$. In particular, $s_1=s_2=l(\nu)$ and   both geodesic segments $\gamma^{(u)}_\nu|_{[0,l(\nu)]}$ and $\beta^{(u)}_\nu|_{[0,l(\nu)]}$ intersect the common point $(\tilde t,\tilde\theta)^{-1}(-u,\varphi(\nu))$, where the number
$l(\nu)$ is defined in \eqref{eq2.3}.
\end{lemma}

\begin{proof} From \cite[Proposition 7.1.3]{SST} it follows that 
    \begin{equation}
        L(\gamma_\nu^{(u)}|_{[0,s_1]})=-\int_{u}^{-\xi(\nu)}g(x,\nu)dx+\int_{-\xi(\nu)}^{-u}g(x,\nu)dx,
    \end{equation}
    where $g(x,\nu):={m(x)}/{\sqrt{m(x)^2-\nu^2}}.$
    Since $g(x,\nu)=g(-x,\nu)$, 
    the length  
       $ L(\gamma^{(u)}_\nu|_{[0,s_1]})$ of $\gamma^{(u)}_\nu|_{[0,s_1]}$
        equals $l(\nu).$
 Applying the similar computation  above to  $\beta^{(u)}_\nu|_{[0,s_2]}$,  
       we get $L(\beta^{(u)}_\nu|_{[0,s_2]})=l(\nu).$
    Therefore, $\gamma^{(u)}_\nu|_{[0,s_1]}$ and $\beta_\nu^{(u)}|_{[0,s_2]}$ 
    have the same length.
    $\qedd$
\end{proof}

\begin{lemma}\label{lem4.3}
    If there exists a geodesic segment $\alpha$ joining the point $q\in\tilde t^{-1}(u)\cap\tilde\theta^{-1}(0)$ to a point in $p\in\tilde t^{-1}(-u)\cap\tilde\theta^{-1}((\varphi(m(u)),\infty)),$ then $\alpha=\gamma_{\nu}^{(u)}$ or $\beta_{\nu}^{(u)}$ for some $\nu\in( m(t_0),m(u)]$. 
\end{lemma}
\begin{proof}
By assuming that the Clairaut constant 
$\nu_1$ of $\alpha$ is not greater than $m(t_0),$
 we will get a contradiction. By the assumption,
 $\alpha$ cannot be tangent to any parallel arc. 
Thus $u\ne0,$ since $\alpha$ is a geodesic segment joining $q\in\tilde t^{-1}(0)$ to $ p\in\tilde t^{-1}(0)$ when $u=0.$
Therefore, we may assume that 
$u\in(-t_0,t_0)\setminus \{0\}$ and 
   $\nu_0:=m(u)> m(t_0).$ 
   If we notice that  $\alpha$ is not tangent to any parallel arc, we have, 
by \cite[Proposition 7.1.3]{SST},
\begin{align}\label{eq4.5}
       \varphi(\nu_0)< \tilde\theta(p)=\int_{-|u|}^{|u|}f(x,\nu_1)dx
        &=\int^{\xi(\nu_0)}_{-\xi(\nu_0)}f(x,\nu_1)dx.
        \end{align}
Since
    $\nu_1\leq m(t_0)< m(u)=\nu_0$, and 
    the function $f(x,\nu) $ is strictly increasing in $\nu\in(0,\nu_0)$ for each $x,$    we get, by \eqref{eq4.5},
     $\varphi(\nu_0)< \int^{\xi(\nu_0)}_{-\xi(\nu_0)}f(x,\nu_1)dx<\varphi(\nu_0),$
   which is a contradiction. 
$\qedd$
\end{proof}

\begin{proposition}\label{prop4.4}
    For each $\nu\in( m(t_0),m(u))$,  $\gamma_\nu^{(u)}|_{[0,l(\nu)]}$ and $\beta_\nu^{(u)}|_{[0,l(\nu)]}$ are minimal geodesic segments joining $q\in\tilde\theta^{-1}(0)$ to the point $(\tilde t,\tilde\theta)^{-1}(-u,\varphi(\nu)).$ In particular, the subarc $\{p\in\widetilde M|\tilde t(p)=-u, \tilde\theta(p)\geq\varphi(m(u))\}$ of the parallel arc $\tilde t=-u$ is a subset of $C_{q}$ and $\gamma_{m(u)}^{(u)}|_{[0,l(m(u))]}=\beta_{m(u)}^{(u)}|_{[0,l(m(u))]}$ is minimal. Here $l(m(0)):=\lim_{\nu \uparrow m(0)}l(\nu)$ and $\varphi(m(0)):=\lim_{\nu \uparrow m(0)}\varphi(\nu),$ when $u=0.$
\end{proposition}

\begin{proof}
      Suppose that for some $\nu_0\in( m(t_0),m(u)),$  $\gamma_{\nu_0}^{(u)}|_{[0,l(\nu_0)]}$ is not a minimal geodesic
      segment joining $q$ to $x:=\gamma_{\nu_0}^{(u)}(l(\nu_0))$. 
      We may assume that $\nu_0=\min\{\nu\in( m(t_0),m(u))|\varphi(\nu)=\varphi(\nu_0)\}$. 
      Let $\alpha:[0,d(q,x)]\to \widetilde M$ denote a minimal geodesic segment joining $q$ to $x.$ 
      By Lemma \ref{lem4.3},
      $\alpha=\gamma_{\nu_1}^{(u)}$  or
      $\beta_{\nu_1}^{(u)}$
       for some $\nu_1\in( m(t_0),m(u)]$.
       Thus, we get 
        $\varphi(\nu_1)=\tilde\theta(x)=\varphi(\nu_0),$ 
        and
          \begin{equation}\label{eq:prop4.4(1)}
         L(\alpha)=l(\nu_1)<l(\nu_0).
         \end{equation}
      Since $\varphi$ is decreasing, $\varphi\equiv\varphi(\nu_0)$ on $[\nu_0,\nu_1]$ and by Lemma \ref{lem2.1},
       $l'(\nu)=\nu\varphi'(\nu)=0$ on $[\nu_0,\nu_1]$. 
       This contradicts \eqref{eq:prop4.4(1)}.
       Since $\gamma_{m(u)}^{(u)}|_{[0,l(m(u))]}=\beta_{m(u)}^{(u)}|_{[0,l(m(u))]}$ is the limit geodesic segment of 
       minimal geodesic segments $\gamma_{\nu}^{(u)}|_{[0,l(\nu)]}$
       as $\nu$ tends to $m(u),$  it is clear that $\gamma_{m(u)}^{(u)}|_{[0,l(m(u))]}=\beta_{m(u)}^{(u)}|_{[0,l(m(u))]}$ is  minimal.
$\qedd$
\end{proof}

\begin{proposition}\label{prop4.5}
    For any point  $q\in\tilde\theta^{-1}(0)\cap\tilde t^{-1}\left((-t_0,t_0)\right)$ the cut locus $C_q$ of the point $q$ is
      $$C_q=\{p\in\widetilde M|\tilde t(p)=-u,|\tilde\theta(p)|\geq\varphi(m(u))\},$$
        where $u:=\tilde t (q).$
\end{proposition}

\begin{proof}
   If we note that $\widetilde M$ is reflective symmetric with respect to the meridian $\tilde\theta=0,$ from Proposition \ref{prop4.4} it follows that the set 
   $$A_u:=\{p\in\widetilde M|\tilde t(p)=-u,|\tilde\theta(p)|\geq\varphi(m(u))\}$$
   is a subset of $C_q.$
    Next, we prove that $C_q$ is a subset of $A_u$. By supposing that there exists a cut point $y$ of $q$ such that $y\notin A_u,$
    we will get a contradiction.
    Without loss of generality, we may assume that $\tilde\theta(y)>0$ and $y$ is an end point of $C_q,$ since $C_q$ is a local tree \cite[Theorem 4.2.1]{SST}. 
     Since $\gamma_{\nu}^{(u)}|_{[0,l(\nu)]}$ and $\beta_{\nu}^{(u)}|_{[0,l(\nu)]}$ are minimal for each $\nu\in( m(t_0),m(u))$, $\gamma_{ m(t_0)}^{(u)}$ and $\beta^{(u)}_{ m(t_0)}$ are also minimal. If we note  that $\lim_{\nu\to m(t_0)}l(\nu)=\infty$,
    both $\gamma_{m(t_0)}^{(u)}$ and $\beta_{ m(t_0)}^{(u)}$ are rays, that is, any subarcs of finite length  are minimal. 
    The point $y$ is not in the closed domain cut off by two rays $\gamma_{ m(t_0)}^{(u)} $ and $\beta_{m(t_0)}^{(u)}.$ 
    This means that $y$ is a point of the domain $D^{+}$ or $D^{-}.$
    Here $D^{+}$ (respectively $D^{-}$)
    denotes the domain cut off by two rays
   $\beta_{m(t_0)}^{(u)}$ and  $\beta_{0}^{(u)}$ (respectively $\gamma_{ m(t_0)}^{(u)}$ and  $\gamma_{0}^{(u)}$).
   Choose any minimal geodesic segment $\alpha$ joining $q$ to $y,$ which is conjugate to $q$ along $\alpha,$ since $y$ is an end point of $C_q$. 
    By applying the Clairaut relation at
     $q$, we get
    the Clairaut constant $\nu_1$ of $\alpha$ is less than $ m(t_0).$
    Thus, $\alpha$ cannot be tangent to any parallel arc $\tilde t=c$ and by \cite[ Proposition 7.2.2 ]{SST}, we can conclude that $y$ is not a conjugate point of $ q$ along $\alpha,$ which is a contradiction. 
    $\qedd$
\end{proof}

\begin{proposition}\label{prop4.6}
Suppose that $t_0<\infty.$ Then
for any point $ q\in\tilde\theta^{-1}(0)$ with
$|\tilde t(q)|\geq t_0$,
the
 cut locus of $q$ 
    is empty.
\end{proposition}

\begin{proof}
We may assume that $\tilde t(q)\geq t_0.$
Choose any geodesic $\gamma:[0,\infty)\to \widetilde M$ emanating from $q.$
If the image  of $\gamma$ is a subset of $\tilde t^{-1}([t_0,\infty)),$ then $\gamma$ has no conjugate point of $q,$  since the Gaussian curvature is nonpositive on $\tilde t^{-1}([t_0,\infty)).$ If the geodesic pass through $\tilde t=t_0,$ then its Clairaut constant is less than $m(t_0).$ 
Thus, $\gamma$ is not tangent to any parallel and by \cite[Proposition 7.2.2]{SST} it has no conjugate point of $q.$
 Therefore, any geodesic emanating from $q$ has no conjugate point of $q.$ This implies that $C_q$ is empty.  
    $\qedd$
\end{proof}

From Propositions \ref{prop4.5} and \ref{prop4.6}, we have proved  {\bf Theorem B.}

\section*{Acknowledgement}
The author wishes to express his sincere gratitude to Professor Minoru Tanaka for  his insightful suggestions.

\vspace{3cm}
Nathaphon Boonnam, Sc.D., SFHEA\\
Faculty of Science and Industrial Technology\\
Prince of Songkla University\\
Suratthani 84000 THAILAND\\
e–mail : nathaphon.b@psu.ac.th

\end{document}